\documentclass[12pt,a4paper]{article}

\usepackage{latexsym}
\usepackage{amsmath}
\usepackage{amssymb}
\usepackage{arydshln}

\usepackage{cite}
\usepackage{stmaryrd}
\usepackage{enumerate}

\usepackage{hyperref}

\usepackage{color}
\usepackage{lineno}
\usepackage{graphicx}
\usepackage{ae}
\usepackage{amsmath}
\usepackage{amssymb}
\usepackage{latexsym}
\usepackage{url}
\usepackage{epsfig}
\usepackage{mathrsfs}
\usepackage{amsfonts}
\usepackage{amsthm}

\usepackage{float}
\usepackage{subfig}
\usepackage{tikz}
\usetikzlibrary{positioning}

\newtheorem{theorem}{Theorem}[section]

\newtheorem{lemma}[theorem]{Lemma}

\newtheorem{cor}[theorem]{Corollary}

\newtheorem{example}{Example}
\theoremstyle{definition}
\newtheorem{definition}{Definition}

\numberwithin{equation}{section}
\allowdisplaybreaks

\def\qed{\hfill$\Box$\vspace{12pt}}

\long\def\delete#1{}

\usetikzlibrary{decorations.markings}

\tikzstyle{vertex}=[circle, draw, inner sep=0pt, minimum size=6pt]
\tikzstyle{directed}=[postaction={decorate,
	decoration={markings,mark=at position 0.3 with {\arrow{stealth}}}
}]

\usepackage{xcolor}
\usepackage[normalem]{ulem}

\begin{document}

\title {Fractional revival on oriented Cayley and semi-Cayley graphs over abelian groups}

\author{Ming Jiang$^{a,b}$,~Xiaogang Liu$^{a,b,}$\thanks{Supported by the National Natural Science Foundation of China (No. 12371358).}~$^,$\thanks{ Corresponding author. Email addresses: mjiang@mail.nwpu.edu.cn, xiaogliu@nwpu.edu.cn,
jingwang\_math@snnu.edu.cn}~,~Jing Wang$^{c}$
\\[2mm]
{\small $^a$School of Mathematics and Statistics,}\\[-0.8ex]
{\small Northwestern Polytechnical University, Xi'an, Shaanxi 710072, P.R.~China}\\
{\small $^b$Xi'an-Budapest Joint Research Center for Combinatorics,}\\[-0.8ex]
{\small Northwestern Polytechnical University, Xi'an, Shaanxi 710129, P.R. China}\\
{\small $^c$School of Mathematics and Statistics,}\\[-0.8ex]
{\small Shaanxi Normal University, Xi'an, Shaanxi 710119, P.R. China} 
}
\date{}

\openup 0.5\jot
\maketitle

\begin{abstract}
Fractional revival (FR), a generalization of perfect state transfer (PST), is a significant phenomenon in quantum state transfer that allows quantum information to be transmitted between two qubits with a certain probability. The existence of FR has been extensively studied on many classes of graphs. However, oriented graphs have not yet been investigated. In this paper, we investigate the existence of proper FR on oriented graphs. We first establish necessary and sufficient conditions for oriented graphs to admit proper FR between strongly cospectral vertices. Furthermore, we prove that oriented Cayley graphs over abelian groups do not admit proper FR, and we subsequently characterize the conditions under which oriented semi-Cayley graphs over abelian groups admit proper FR.
\smallskip



\end{abstract}

\section{Introduction}

Let $\Gamma$ be an undirected graph with vertex set $V(\Gamma)$ and edge set $E(\Gamma) \subseteq V(\Gamma) \times V(\Gamma) $. The \emph{adjacency matrix} of $\Gamma$, denoted by $A_{\Gamma}$, is an $n \times n$ matrix whose $(u,v)$-entry is $1$ if $(u,v) \in E(\Gamma)$, and $0$ otherwise. Note that $(u,v) \in E(\Gamma)$ if and only if $(v,u) \in E(\Gamma)$; hence $A_{\Gamma}$ is Hermitian. The transition matrix of the \emph{continuous-time quantum walk} on $\Gamma$ with Hamiltonian $A_{\Gamma}$ is given by
\begin{equation}\label{transM11}
U(t) = \exp(-\mathrm{i} A_{\Gamma} t) = \sum_{k=0}^{\infty} \frac{(-\mathrm{i})^k A_{\Gamma}^k t^k}{k!}, \qquad t \in \mathbb{R}, \ \mathrm{i} = \sqrt{-1}.
\end{equation}
As one of the universal models of quantum computing, continuous-time quantum walks find broad applications in quantum communication, network analysis and quantum algorithms, and have attracted substantial attention from researchers.


In 2003, Bose first investigated quantum state transfer on quantum spin chains via continuous-time quantum walks and demonstrated high-fidelity quantum state transfer~\cite{Bose2003}. One year later, Christandl et al. generalized this idea to arbitrary graphs: they abstracted quantum network models as graphs and utilized continuous-time quantum walks on graphs to characterize the time evolution of quantum systems~\cite{Christandl2004}. These pioneering works motivated subsequent research on \emph{perfect state transfer} (PST for short) on diverse families of graphs, such as trees~\cite{God2012}, Cayley graphs~\cite{Basi2013, Basi2009, Cao2020, Cao2021, Cheung2011, Tan2019}, distance-regular graphs~\cite{Coutinho2015} and graph operations~\cite{Pal2016, Ack2017, Ang}.

We now give the definition of PST. Given two distinct vertices $u$ and $v$, let $\mathbf{e}_{u}$ and $\mathbf{e}_{v}$ denote the standard basis vectors indexed by $u$ and $v$, respectively. We say that $\Gamma$ admits PST between $u$ and $v$ if there exists a time $t$ and a complex number $\gamma$ with $|\gamma| = 1$ such that
$$
U(t)\mathbf{e}_{u}=\gamma \mathbf{e}_{v}.
$$
In 2010, Kay proved that if an undirected graph admits PST between $u$ and $v$ and also between $u$ and $w$, then $w=v$ \cite{Kay}.

In contrast to the undirected case, oriented graphs exhibit much richer spectral and quantum transport behavior.
Given a graph $\Gamma$ with vertex set $V(\Gamma)$ and edge set $E(\Gamma)$, an edge $(u,v)$ is called \emph{directed} if $(u,v)\in E(\Gamma)$ but $(v,u)\notin E(\Gamma)$. A graph $\Gamma$ is called \emph{oriented} when all its edges are directed. The adjacency matrix of an oriented graph $\Gamma$, denoted by $A_{\Gamma}=[a_{uv}]$, is defined  entrywise by
$$
a_{uv}=\begin{cases}
 1, & \text{ if } (u,v)\in E(\Gamma), \\
-1, & \text{ if } (v,u)\in E(\Gamma), \\
 0, & \text{ otherwise. }\\
\end{cases}
$$
where $a_{uv}$ denotes the entry in the $u$-th row and $v$-th column of $A_{\Gamma}$. Since $A_{\Gamma}$ is real skew-symmetric, we take $\mathrm{i}A_{\Gamma}$ as the Hamiltonian for the continuous quantum walk on $\Gamma$. The \emph{transition matrix} of the continuous quantum walk on $\Gamma$ with respect to $\mathrm{i}A_{\Gamma}$ is given by
\begin{equation}\label{transM}
U(t)=\exp(-\mathrm{i}(\mathrm{i}A_{\Gamma}t))=\exp(A_{\Gamma}t),~ t \in \mathbb{R},~\mathrm{i}=\sqrt{-1}.
\end{equation}
Clearly, $U(t)$ is always a real unitary matrix.

Recall that in undirected graphs, a vertex can only admit PST with at most one other vertex; this restriction no longer holds for oriented graphs, where multiple pairwise PST among a set of vertices becomes possible. In 2014, Cameron et al. introduced the notion of \emph{universal perfect state transfer} (short for UPST),  meaning PST exists between every pair of vertices in the whole graph. They further conjectured that the oriented 3-cycle and the  single oriented edge are the only graphs admitting UPST~\cite{Came2014}. This conjecture inspired a stream of follow-up research. In 2017, Connelly gave new characterizations of graphs with UPST~\cite{Conn2017}. In 2020, Godsil and Lato extended PST theory from undirected graphs to oriented graphs, derived necessary and sufficient conditions for PST on oriented graphs, and generalized UPST by proposing \emph{multiple perfect state transfer} (MPST for short): a graph is said to admit MPST if there exists a  vertex subset $S$ with $\left|S\right|\ge 3$ such that PST occurs between any two distinct  vertices in $S$~\cite{God2020}. In 2023, Acuaviva verified Cameron's 2014 conjecture and also constructed oriented graphs admitting MPST with $\left|S\right|=4$~\cite{Acuaviva2023}. Most recently in 2024, Chen et al. constructed families of weighted oriented graphs with UPST via bipartite walks~\cite{Chen2024}. In the same year, Song investigated the existence of PST on integral oriented circulant graphs~\cite{song}.

Although PST on oriented graphs has been systematically investigated, \emph{fractional revival} (FR for short) --- another key quantum transport phenomenon generalized from PST --- remains unexplored over oriented graphs. A series of fundamental results concerning FR on undirected graphs have been established, which both motivate our work and supply essential technical preliminaries.

As a natural generalization of PST, FR was first observed and analyzed in quantum optics and spin chain systems. In 2015, Rohith et al. introduced FR as a generalization of PST in quantum state transmission~\cite{Rohith}. Subsequently, Genest et al. presented a systematic study of FR at two sites in XX quantum spin chains~\cite{Genest2016}. In the context of continuous quantum walks on graphs, Chan et al. provided the first systematic treatment of FR in 2019, characterizing the necessary and sufficient conditions for an undirected graph to admit FR in terms of its spectral decomposition~\cite{Chan2019}. This seminal work stimulated a series of subsequent studies. In 2020, Chan et al. investigated the existence of FR in graphs whose adjacency matrices belong to the Bose-Mesner algebra of association schemes~\cite{Chan2020}. In 2021, Chan  et al. developed FR theory by using the Laplacian matrix as the Hamiltonian for the quantum walk~\cite{Chan2021}. In 2023, Monterde provided a characterization of FR between twin vertices in weighted graphs~\cite{Mon2023}. Most recently, Wang  et al.  gave a necessary and sufficient condition for Cayley graphs over finite abelian groups to have FR~\cite{WangJ2024} and further demonstrated the existence of FR on semi-Cayley graphs over finite abelian groups~\cite{WangJ2025}.

Against this background, the present paper initiates the first investigation of FR on oriented graphs and establishes a spectral characterization of FR in this setting. Given two distinct vertices $u$ and $v$ in oriented graph $\Gamma$, we say that $\Gamma$ admits $(\alpha, \beta)$-FR with $\beta\neq 0$ from $u$ to $v$ if there exists a time $t$ such that
\begin{equation}\label{FR1}
U(t)\mathbf{e}_{u}=\alpha \mathbf{e}_{u} +\beta \mathbf{e}_{v},
\end{equation}
where $\alpha, \beta\in \mathbb{R} $ and satisfy
$$
\alpha ^2 +\beta ^2=1.
$$
When $\alpha\beta\neq0$, the FR is said to be \emph{proper}.

Our main contributions are outlined below. In Section \ref{Sec-222}, we establish necessary and sufficient conditions for oriented graphs to admit proper FR between strongly cospectral vertices. In Section \ref{Sec-333}, we prove that oriented Cayley graphs over abelian groups do not admit proper FR. In Section \ref{Sec-444}, we derive necessary and sufficient conditions for oriented semi-Cayley graphs over abelian groups to admit proper FR.

\section{Preliminaries}\label{Sec-222}

For an oriented graph $\Gamma$ of order $n$, let $\lambda_1,\dots,\lambda_n$ be the eigenvalues of  $A_\Gamma$ with corresponding orthonormal eigenvectors $\mathbf{x}_1,\dots,\mathbf{x}_n$. For each $r\in \{1,\dots,n\}$, define the eigenprojector associated with $\lambda_r$ by
$$
E_{\lambda_{r}} = \mathbf{x}_r \left(\mathbf{x}_r\right)^H,
$$
where $\ast^H$ denotes the conjugate transpose of $\ast$. Notice that $\sum_{r=1}^{n}E_{\lambda_{r}}=I$, where $I$ is the identity matrix of order $n$. The spectral decomposition of $A_\Gamma$ takes the form
\begin{equation}\label{decom1}
A_\Gamma=A_\Gamma\sum_{r=1}^{n}E_{\lambda_{r}}=\sum_{r=1}^{n}\lambda_{r}E_{\lambda_{r}}.
\end{equation}
Note that $E_{\lambda_{r}}E_{\lambda_{r}}=E_{\lambda_{r}}$ and $E_{\lambda_{r}}E_{\lambda_{s}}=0$ for each $r\neq s$. It follows from \eqref{transM} and \eqref{decom1} that
\begin{equation}\label{Udecompose}
U(t)=\sum_{k\ge 0}\frac{A_{\Gamma}^{k}t^{k}}{k!}=\sum_{r=1}^{n}\exp(t\lambda_{r})E_{\lambda_r}.
\end{equation}

Given two distinct vertices $u$ and $v$ in an oriented graph $\Gamma$, we say that $u$ and $v$ are \emph{cospectral} if for each $r\in \{1,\dots,n\}$,
$$
(E_{\lambda_{r}})_{u,u}=(E_{\lambda_{r}})_{v,v},
$$
where $(E_{\lambda_{r}})_{u,u}$ denotes the $(u,u)$-entry of $E_{\lambda_{r}}$. This definition coincides with the one for undirected graphs. Recall that $A_\Gamma$ is skew symmetric. Hence its eigenvalues are purely imaginary or zero, and its eigenvectors are generally complex. In this oriented setting, vertices $u$ and $v$ are said to be \emph{strongly cospectral} if for every $r\in \{1,\dots,n\}$, there exists a phase factor $e^{\mathrm{i}\pi q_r(u,v)}$ such that
\begin{equation}\label{Scospectral}
E_{\lambda_{r}}\mathbf{e}_{u}=e^{\mathrm{i}\pi q_r(u,v)} E_{\lambda_{r}}\mathbf{e}_{v},
\end{equation}
with the condition that these phase factors depend only on the eigenvalue:
\begin{equation}\label{laml=lam2}
\lambda_i=\lambda_j ~\Rightarrow ~ e^{\mathrm{i}\pi q_i(u,v)}=e^{\mathrm{i}\pi q_j(u,v)}.
\end{equation}
For each vertex $u$, let $\phi_u$ denote the set of eigenvalues $\lambda_{r}$ for which $E_{\lambda_{r}}\mathbf{e}_{u}\neq \mathbf{0}$;  we refer to $\phi_u$ as the \emph{eigenvalue support} of $u$.

\begin{lemma}\label{Cospectral-111}
Suppose that two distinct vertices $u$ and $v$ are strongly cospectral in an oriented graph $\Gamma$. Then $u$ and $v$ are cospectral in $\Gamma$.
\end{lemma}
\begin{proof}
If $u$ and $v$ are strongly cospectral, then for each $r=1,\dots,n$, we have
$$
E_{\lambda_{r}}\mathbf e_{u}=e^{\mathrm{i}\pi q_r(u,v)} E_{\lambda_{r}}\mathbf e_{v}.
$$
It follows that
$$
(E_{\lambda_{r}}\mathbf e_{u})^{H}E_{\lambda_{r}}\mathbf e_{u}=(e^{\mathrm{i}\pi q_r(u,v)} E_{\lambda_{r}}\mathbf e_{v})^{H}e^{\mathrm{i}\pi q_r(u,v)} E_{\lambda_{r}}\mathbf e_{v}.
$$
Since eigenprojectors are idempotent and Hermitian, we obtain that $(E_{\lambda_{r}})_{u,u}=(E_{\lambda_{r}})_{v,v}$ for each $r=1,\dots,n$, that is, $u$ and $v$ are cospectral in $\Gamma$.
\qed
\end{proof}

\begin{lemma}\label{Cospectral}
Suppose that two distinct vertices $u$ and $v$ are (strongly) cospectral in an oriented graph $\Gamma$. Then $U(t)_{u,u}=U(t)_{v,v}$ at any time $t$, where $U(t)_{u,u}$ denotes the $(u,u)$-entry of the transition matrix $U(t)$.
\end{lemma}
\begin{proof}
If $u$ and $v$ are strongly cospectral, by Lemma \ref{Cospectral-111} they are cospectral. It therefore suffices to consider the cospectral case.

If $u$ and $v$ are cospectral, then $(E_{\lambda_{r}})_{u,u}=(E_{\lambda_{r}})_{v,v}$ for each $r=1,\dots,n$. By \eqref{Udecompose}, we have $U(t)_{u,u}=U(t)_{v,v}$ at any time $t$.
\qed
\end{proof}

\begin{theorem}\label{FR=PST}
Suppose that two distinct vertices $u$ and $v$ are (strongly) cospectral in an oriented graph $\Gamma$. If $(\alpha, \beta)$-FR with $\alpha, \beta\neq 0$ occurs at time $t$ from $u$ to $v$, then $(\alpha, -\beta)$-FR occurs from $v$ to $u$ at the same time $t$.
\end{theorem}
\begin{proof}
Suppose that $(\alpha, \beta)$-FR occurs from $u$ to $v$ at time $t$. Then
\begin{equation}\label{FRab}
U(t)\mathbf e_{u}=\alpha\mathbf e_{u}+ \beta\mathbf e_{v},
\end{equation}
with $\alpha ^2 +\beta ^2=1 $. By Lemma \ref{Cospectral} and (\ref{FRab}), we get $U(t)_{v,v}=U(t)_{u,u}=\alpha$.

Note that $U(t)$ is unitary. It follows that
$$
(U(t)\mathbf e_{u})^{H}(U(t)\mathbf e_{v})=0.
$$
Substituting \eqref{FRab} into the above equation, we have
\begin{equation*}
(\alpha\mathbf e_{u}+ \beta\mathbf e_{v})^{H}(U(t)\mathbf e_{v})=\alpha U(t)_{u,v}+ \beta U(t)_{v,v}=0.
\end{equation*}
Thus, $U(t)_{u,v}=-\beta$. Moreover, since $U(t)$ is unitary,
$$
1=\|U(t)e_v\|^2\geq |U(t)_{u,v}|^2+|U(t)_{v,v}|^2=\beta^2+\alpha^2=1.
$$
Therefore, equality holds and all the remaining entries in the $v$-th column of
$U(t)$ are zero. Consequently,
$$
U(t)\mathbf e_{v}=\alpha \mathbf e_{v}-\beta \mathbf e_{u}.
$$

Therefore, we conclude that $(\alpha,-\beta)$-FR occurs at time $t$ from $v$ to $u$.
\qed
\end{proof}

In the following, we provide a characterization of oriented graphs admitting FR between strongly cospectral vertices.

\begin{theorem}\label{FRNS}
Let $u$ and $v$ be strongly cospectral vertices in an oriented graph $\Gamma$. Then $\Gamma$ admits $(\alpha, \beta)$-FR with $\alpha, \beta\neq0$ from $u$ to $v$ at time $t$ if and only if for each $\lambda_{r}\in \phi_u$,
$$
\mathrm{Re} \left(e^{\lambda_{r} t}\right)=\alpha,~ ~\mathrm{Im} \left(e^{\lambda_{r} t}\right)=\pm \beta,
$$
and
\begin{equation}\label{ipm}
e^{\mathrm{i}\pi q_r(u,v)} = -\frac{\mathrm{Im}\left(e^{\lambda_r t}\right)}{\beta}\,\mathrm{i}.
\end{equation}
\end{theorem}
\begin{proof}
Suppose that $(\alpha, \beta)$-FR occurs from $u$ to $v$ at time $t$. By Theorem \ref{FR=PST}, we have
\begin{equation}\label{FRabba}
\left\{\begin{matrix}
U(t)\mathbf e_{u}=\alpha\mathbf e_{u}+ \beta\mathbf e_{v},     \\[0.2cm]
U(t)\mathbf e_{v}=\alpha\mathbf e_{v}- \beta\mathbf e_{u}.
\end{matrix}\right.
\end{equation}
For each $\lambda_{r}\in \phi_u$, multiplying both sides of \eqref{FRabba} on the left by $E_{\lambda_{r}}$ yields
\begin{equation}\label{FRcosab}
\left\{\begin{matrix}
\exp(\lambda_{r}t)E_{\lambda_{r}}\mathbf e_{u}=\alpha E_{\lambda_{r}}\mathbf e_{u}+ \beta E_{\lambda_{r}}\mathbf e_{v},     \\[0.2cm]
\exp(\lambda_{r}t)E_{\lambda_{r}}\mathbf e_{v}=\alpha E_{\lambda_{r}}\mathbf e_{v}- \beta E_{\lambda_{r}}\mathbf e_{u}.
\end{matrix}\right.
\end{equation}
Since $u$ and $v$ are strongly cospectral, there exists a complex number $e^{\mathrm{i}\pi q_r(u,v)}$ such that
\begin{equation}\label{FRcosab1}
E_{\lambda_{r}}\mathbf e_{u}=e^{\mathrm{i}\pi q_r(u,v)} E_{\lambda_{r}}\mathbf e_{v}.
\end{equation}
Combining (\ref{FRcosab}) with (\ref{FRcosab1}), we have
\begin{equation}\label{FRcosab1-1-1-2}
e^{\mathrm{i}\pi q_r(u,v)} = \overline{\frac{e^{\lambda_{r}t}-\alpha}{\overline{\beta}}} = -\frac{e^{\lambda_{r}t}-\alpha}{\beta},
\end{equation}
where $\overline{\ast}$ denotes the complex conjugate of $\ast$. It follows that $\mathrm{Re} (e^{\lambda_{r} t})=\alpha$. Note that $\left|e^{\lambda_{r} t}\right|=1$ and $\alpha ^2 +\beta ^2=1$. Then $\mathrm{Im} (e^{\lambda_{r} t})=\pm \beta$. By \eqref{FRcosab1-1-1-2}, we obtain
$$
e^{\mathrm{i}\pi q_r(u,v)}=-\frac{\mathrm{Im}\left(e^{\lambda_r t}\right)}{\beta}\,\mathrm{i}.
$$

Conversely, suppose that for every $\lambda_r\in\phi_u$,
$$
\operatorname{Re}(e^{\lambda_r t})=\alpha,\quad \operatorname{Im}(e^{\lambda_r t})=\pm\beta,\quad\text{and}\quad E_{\lambda_r}\mathbf e_u = \mp\mathrm{i}\,E_{\lambda_r}\mathbf e_v.
$$
Partition $\phi_u$ into
$$
\Lambda_{u,v}^+ = \{\lambda_r \in \phi_u: E_{\lambda_r}\mathbf e_u = +\mathrm{i}E_{\lambda_r}\mathbf e_v\},\qquad
\Lambda_{u,v}^- = \{\lambda_r  \in \phi_u: E_{\lambda_r}\mathbf e_u = -\mathrm{i}E_{\lambda_r}\mathbf e_v\}.
$$
 Then
$$
e^{\lambda_r t} =\left\{
\begin{array}{cc}
  \alpha-\mathrm{i}\beta, & \text{if $\lambda_r\in\Lambda_{u,v}^+$}, \\[0.2cm]
  \alpha+\mathrm{i}\beta, &  \text{if $\lambda_r\in\Lambda_{u,v}^-$}.
\end{array}
\right.
$$
Note that $\sum_{r=1}^{n}E_{\lambda_r}=I$ and $E_{\lambda_r}\mathbf e_u=0$ for $\lambda_r\notin\phi_u$. By \eqref{Udecompose}, we have
\begin{equation*}
\begin{aligned}
U(t)\mathbf{e}_{u}
&=\sum_{r=1}^{n}\exp(\lambda_{r}t)E_{\lambda_{r}}\mathbf{e}_{u}\\
&=\sum_{\lambda_{r}\in \Lambda_{u,v}^+}\exp(\lambda_{r}t)E_{\lambda_{r}}\mathbf{e}_{u}+\sum_{\lambda_{r}\in \Lambda_{u,v}^-}\exp(\lambda_{r}t)E_{\lambda_{r}}\mathbf{e}_{u}\\
&=\sum_{\lambda_{r}\in \Lambda_{u,v}^+}(\alpha-\mathrm{i}\beta)E_{\lambda_{r}}\mathbf{e}_{u}+\sum_{\lambda_{r}\in \Lambda_{u,v}^-}(\alpha+\mathrm{i}\beta)E_{\lambda_{r}}\mathbf{e}_{u}\\
&= \alpha\sum_{\lambda_r\in\phi_u}E_{\lambda_r}\mathbf e_u + \mathrm{i}\beta\left(-\sum_{\lambda_r\in\Lambda_{u,v}^+}E_{\lambda_r}\mathbf e_u + \sum_{\lambda_r\in\Lambda_{u,v}^-}E_{\lambda_r}\mathbf e_u\right)\\
&=\alpha\mathbf e_u + \mathrm{i}\beta (-\mathrm{i}\mathbf e_v)\\
&=\alpha \mathbf{e}_{u}+\beta \mathbf{e}_{v}.
\end{aligned}
\end{equation*}
Therefore, $\Gamma$ admits $(\alpha,\beta)$-FR from $u$ to $v$ at time $t$.
\qed
\end{proof}

Let $u$ and $v$ be strongly cospectral vertices in an oriented graph $\Gamma$. If $\Gamma$ admits $(\alpha, \beta)$-FR with $\alpha, \beta\neq0$ from $u$ to $v$ at time $t$, by \eqref{ipm}, we see that $e^{\mathrm{i}\pi q_r(u,v)} = \pm\mathrm{i}$ for all $\lambda_{r}\in \phi_u$. For later use, we restate these two important subsets of the eigenvalues subject to this condition as follows:
\begin{align*}
 \Lambda_{u,v}^+ &= \{\lambda_r\in \phi_u : E_{\lambda_r}\mathbf e_u = +\mathrm{i}E_{\lambda_r}\mathbf e_v\},\\
\Lambda_{u,v}^- &= \{\lambda_r\in \phi_u : E_{\lambda_r}\mathbf e_u = -\mathrm{i}E_{\lambda_r}\mathbf e_v\}.
\end{align*}

\begin{theorem}\label{Epm}
Let $u$ and $v$ be strongly cospectral vertices in an oriented graph $\Gamma$.  Then
$$\lambda \in \Lambda^+_{u,v}~\text{if and only if}~-\lambda \in \Lambda^-_{u,v}.$$
\end{theorem}
\begin{proof}
Since the adjacency matrix $A_\Gamma$ is real skew-symmetric, its eigenvalues are either zero or purely imaginary. Thus, if $\lambda$ is an eigenvalue, then $-\lambda = \overline{\lambda}$ is also an eigenvalue. Moreover, the corresponding eigenprojectors satisfy $E_{-\lambda} = \overline{E_\lambda}$.

Suppose that $\lambda \in \Lambda^+_{u,v}$, i.e., $E_\lambda\mathbf e_u = \mathrm{i} E_\lambda\mathbf e_v$.
Taking complex conjugation on both sides, we obtain
$$
E_{-\lambda}\mathbf e_u = -\mathrm{i} E_{-\lambda}\mathbf e_v,
$$
which means $-\lambda \in \Lambda^-_{u,v}$.

The converse direction is similar and therefore omitted.
\qed
\end{proof}

\begin{cor}\label{FRNS2}
Let $u$ and $v$ be strongly cospectral vertices in an oriented graph $\Gamma$. Suppose that $ \phi_u=\Lambda^+_{u,v}\cup\Lambda^-_{u,v}$ and fix $\lambda\in\Lambda^+_{u,v} $. Then $\Gamma$ admits $(\alpha, \beta)$-FR with $\alpha, \beta\neq0$ from $u$ to $v$ at time $t$ if and only if
$$\mathrm{Re}(e^{\lambda t})=\alpha,\qquad\mathrm{Im}(e^{\lambda t})=-\beta, $$
and
$$(\lambda_r - \lambda)t \in 2\pi\mathrm{i}\mathbb{Z} ~~~\text{for each}~ \lambda_{r}\in\Lambda^{+}_{u,v}.$$
\end{cor}
\begin{proof}
By Theorem \ref{FRNS}, $(\alpha, \beta)$-FR  with $\alpha, \beta\neq0$ occurs from $u$ to $v$ at time $t$ if and only if for each $\lambda_{r}\in \phi_{u}$,
$$
e^{\lambda_r t} =\left\{
\begin{array}{cc}
  \alpha-\mathrm{i}\beta, & \text{if $\lambda_r\in\Lambda_{u,v}^+$}, \\[0.2cm]
  \alpha+\mathrm{i}\beta, &  \text{if $\lambda_r\in\Lambda_{u,v}^-$}.
\end{array}
\right.
$$
Consequently, $\mathrm{Re}(e^{\lambda t})=\alpha$, $\mathrm{Im}(e^{\lambda t})=-\beta$. Moreover, for every $\lambda_r\in\Lambda^+_{u,v}$, $e^{\lambda_r t} = e^{\lambda t}$ if and only if
\begin{equation}\label{Lam+1}
(\lambda_r - \lambda)t \in 2\pi\mathrm{i}\mathbb{Z}.
\end{equation}
Theorem \ref{Epm} implies that $-\lambda\in\Lambda^-_{u,v} $. Then for every $\lambda_r\in\Lambda^-_{u,v}$, $e^{\lambda_r t} = e^{-\lambda t}$ if and only if
\begin{equation}\label{Lam-1}
(\lambda_r + \lambda)t \in 2\pi\mathrm{i}\mathbb{Z}.
\end{equation}

According to Theorem \ref{Epm}, we have $\Lambda^{+}_{u,v}=-\Lambda^{-}_{u,v}$. It follows that \eqref{Lam+1} and \eqref{Lam-1} are equivalent. Therefore, the FR condition reduces to the single condition \eqref{Lam+1}. This completes the proof.
\qed
\end{proof}

\section{FR on oriented Cayley graphs}\label{Sec-333}

Let $G$ be a finite abelian group of order $n$.  It is well known that $G$ can be decomposed as a direct sum of cyclic groups:
$$
G = \mathbb{Z}_{n_1} \oplus \mathbb{Z}_{n_2} \oplus \cdots \oplus \mathbb{Z}_{n_p},
$$
where $n_s \ge 2$, and $\mathbb{Z}_{m} = (\mathbb{Z}/m\mathbb{Z}, +)$ denotes the cyclic group of order $m$. For an element $z = (z_1,\dots,z_p) \in G$ with $z_s \in \mathbb{Z}_{n_s}$, we define a map
$$
\chi_z : G \to \mathbb{C}, \quad
\chi_z(g) = \prod_{s=1}^p \omega_{n_s}^{z_s g_s} \quad \text{for } g = (g_1,\dots,g_p) \in G,
$$
where $\omega_{n_s} = \exp(2\pi \mathrm{i}/n_s)$ is a primitive $n_s$-th root of unity. The map $\chi_z$ is called a \emph{character} of $G$. Given a character $\chi_z$ and a subset $S \subseteq G$, we write
$$
\chi_z(S) = \sum_{x \in S} \chi_z(x)
$$
to denote the character sum over $S$.

Let $\widehat{G}$ denote the set of all irreducible characters of $G$. There is a natural isomorphism $\phi : G \to \widehat{G}$ given by $\phi(z) = \chi_z$, which satisfies $\chi_z(g) = \chi_g(z)$ for all $z,g \in G$. The trivial character $\chi_0 \in \widehat{G}$ is defined by $\chi_0(g)=1$ for every $g \in G$. Moreover, the first orthogonality relation states that for any $\chi, \psi \in \widehat{G}$,
\begin{equation}\label{OR}
\frac{1}{n} \sum_{g \in G} \chi(g) \overline{\psi(g)} =
\begin{cases}
1, & \chi = \psi,\\[0.2cm]
0, & \chi \ne \psi.
\end{cases}
\end{equation}
In particular, if \(\chi \in \widehat{G}\) is nontrivial (i.e., \(\chi \ne \chi_0\)), then
\begin{equation}\label{CHIsum}
\sum_{g \in G} \chi(g) = \sum_{g \in G} \overline{\chi(g)} = 0.
\end{equation}

Given an abelian group $G$ and a function $f: G \to \mathbb{C}$, the \emph{Cayley color graph} $\mathrm{Cay}(G,f)$ is the oriented graph with vertex set $G$ in which each arc $(x,y)$ is assigned the color $f(x^{-1}y)$ (where $x^{-1}$ denotes the inverse of $x$). The eigenvalues of $\mathrm{Cay}(G,f)$ are given as follows.

\begin{lemma}\emph{(See  \cite[Corollary~3.2] {Babai})} \label{color}
Let $G$ be a finite abelian group and $f: G \to \mathbb{C}$ a connection function. Then the spectrum of $\mathrm{Cay}(G,f)$ is given by $\{\lambda_{k}\mid k\in G\}$, where
$$
\lambda_{k}=\sum_{y\in G}f(y)\chi_{k}(y),~~\text{for all $k\in G$}.
$$
\end{lemma}

Let $S$ be a subset of $G$ such that $S\cap S^{-1}=\emptyset $, where $S^{-1}=\{s^{-1}\mid s\in S\}$. Then $\Gamma=\mathrm{Cay}(G,S)$ denotes the \emph{oriented Cayley graph} with vertex set $G$, and there is an oriented arc from $u$ to $v$ if and only if $u^{-1}v\in S$. According to the definition of the adjacency matrix of oriented graph, we define a function $f_S:~G\to \mathbb{R}$ by
$$
f_S(g)=\left\{\begin{array}{cl}
   1, & \text{if $g\in S$},  \\[0.2cm]
-1, & \text{if $g\in S^{-1}$},  \\[0.2cm]
 0, & \text{otherwise}.
              \end{array}\right.
$$
Then the oriented Cayley graph $\mathrm{Cay}(G,S)$ is exactly the Cayley color graph $\mathrm{Cay}(G,f_S)$. By Lemma \ref{color}, the spectrum of $\mathrm{Cay}(G,S)$ is given as follows.

\begin{lemma}\label{oriented}
Let $G$ be a finite abelian group, and let $S\subseteq G$ be a subset satisfying $S\cap S^{-1}=\emptyset $. Then the spectrum of $\mathrm{Cay}(G,S)$ is given by $\{\lambda_{k}\mid k\in G\}$, where
$$
\lambda_{k}=\chi_{k}(S)-\chi_{k}(S^{-1}),~~\text{for all $k\in G$}.
$$
\end{lemma}

For convenience, we denote the oriented Cayley graph $\mathrm{Cay}(G,S)$ by $\Gamma$. Define the character table $P=\frac{1}{\sqrt{\left|G\right|}}(\chi_{g}(h))_{g,h\in G}$ and let $p_k$ be the $k$-th column of $P$. By Lemma \ref{oriented}, we have
\begin{equation*} 
P^{H}A_{\Gamma}P=\mathrm{diag}(\chi_{g}(S)-\chi_{g}(S^{-1}))_{g\in G}.
\end{equation*}
The spectral decomposition of $A_\Gamma$ is written as
\begin{equation}\label{DecomA}
A_\Gamma=\sum_{k\in G}\lambda_{k}\widetilde{E}_{k},
\end{equation}
where
\begin{equation}\label{EX}
\widetilde{E}_{k}=p_k p_k^{H}=\frac{1}{\left|G\right|}(\chi_{k}(gh^{-1}))_{g,h\in G}.
\end{equation}
Since all the characters $\chi\in \widehat{G}$ satisfy the first orthogonality relation, it follows from \eqref{EX} that
\begin{equation}\label{Firstorth}
\widetilde{E}_x\widetilde{E}_y=\begin{cases}
\widetilde{E}_x, & \text{ if } x=y, \\[0.2cm]
0, & \text{ if } x\neq y.
\end{cases}
\end{equation}
By \eqref{transM}, \eqref{DecomA} and \eqref{Firstorth}, the transition matrix of the quantum walk on $\Gamma$ (with respect to $A_{\Gamma}$) can be written as
\begin{equation}\label{UIA}
U(t)=\exp(A_{\Gamma}t)=\sum_{k\in G}\exp(\lambda_{k}t)\widetilde{E}_{k}.
\end{equation}
Then for any vertices $a, b$ of $\Gamma$, by \eqref{EX} and \eqref{UIA}, we have
\begin{equation}\label{Uaa}
U(t)_{a,a}=\frac{1}{n}\sum_{k\in G}\exp(\lambda_{k}t),
\end{equation}
and
\begin{equation}\label{Uab}
U(t)_{b,a}=\frac{1}{n}\sum_{k\in G}\exp(\lambda_{k}t)\chi_{k}(ba^{-1}).
\end{equation}

\begin{theorem}\label{AnonFR }
Let $G$ be a finite abelian group of order $n$, and let $S\subseteq G$ be a subset satisfying $S\cap S^{-1}=\emptyset $. Then $\mathrm{Cay}(G,S)$ does not admit $(\alpha,\beta)$-FR with $\alpha,\beta\neq0$.
\end{theorem}
\begin{proof} Assume that $\operatorname{Cay}(G,S)$ admits $(\alpha,\beta)$-FR with $\alpha,\beta\neq0$ from a vertex $u$ to a vertex $uv$ at time $t$, where $v\in G\setminus\{e\}$ is some non-identity element. By \eqref{FR1} and \eqref{Uaa}, we have
\begin{equation}\label{U1a}
U(t)_{u,u}=U(t)_{uv,uv}=\alpha.
\end{equation}
By \eqref{FR1} and \eqref{Uab}, we have
\begin{equation}\label{U2a}
U(t)_{uv,u}=U(t)_{uv^2,uv}=\beta.
\end{equation}
Combining \eqref{U1a} with \eqref{U2a}, we know that $\mathrm{Cay}(G,S)$ admits $(\alpha, \beta)$-FR from $uv$ to $uv^2$ at the same time $t$.

On the other hand, it follows from \eqref{EX} that all vertices in $\mathrm{Cay}(G,S)$ are cospectral. Hence, by Theorem \ref{FR=PST},
\begin{equation}\label{U3a}
U(t)_{u,uv}=U(t)_{uv,uv^2}=-\beta.
\end{equation}
If $uv^2 \neq u$, then by \eqref{U1a}, \eqref{U2a} and \eqref{U3a},
$$
\left |  U(t)_{uv,u} \right | ^2+\left | U(t)_{uv,uv} \right | ^2+\left |  U(t)_{uv,uv^2}\right | ^2=\alpha^2+2\beta^2>1,
$$
since $\alpha^2+\beta^2=1$ and $\beta\neq 0$. This contradicts the unitarity of $U(t)$. Thus we must have $uv^2 = u$, i.e., $v^2 = e$. Now combining \eqref{U2a} and \eqref{U3a}, we have
$$
U(t)_{uv,u}=\beta=U(t)_{u,uv}=-\beta,
$$
which forces $\beta=0$, contradicting the assumption $\beta\neq0$.
\qed
\end{proof}

\section{FR on oriented semi-Cayley graphs}\label{Sec-444}

In this section, we investigate the existence of FR on \emph{oriented semi-Cayley graphs} over abelian groups.

\begin{definition}
Let $G$ be a finite group with identity $e$, and let $R, L, S\subseteq  G$ be subsets satisfying
$$
R\cap R^{-1}=L\cap L^{-1}=\emptyset \text{~and~} e\notin R\cup L.
$$
The \emph{oriented semi-Cayley graph} $\mathrm{OSC}(G,R,L,S)$ is the oriented graph whose vertex set is
$$
\{(g,0), (g,1)\mid g\in G\},
$$
and whose directed edges are given by
$$
\begin{matrix}
\{((g,0), (h,0))\mid g^{-1}h\in R\}& ~~\text{(edges on the right side)},\\[0.2cm]
\{((g,1), (h,1))\mid g^{-1}h\in L\}& ~~\text{(edges on the left side)},\\[0.2cm]
\{((g,0), (h,1))\mid g^{-1}h\in S\}& ~~\text{(edges from the right side to the left side)}.
\end{matrix}
$$
\end{definition}

Let $\Gamma = \mathrm{OSC}(G,R,L,S)$ be an oriented semi-Cayley graph over an abelian group $G$ of order $n$.
Denote by $A$ and $C$ the adjacency matrices of the oriented Cayley graphs $\mathrm{Cay}(G,R)$ and $\mathrm{Cay}(G,L)$, respectively.
Define the matrix $B \in \mathbb{R}^{n \times n}$, with rows and columns indexed by the elements of $G$, via
$$
B_{g,h} = \begin{cases}
1, & \text{if } g^{-1}h\in S,\\[0.2cm]
0, & \text{otherwise}.
\end{cases}
$$
Then the adjacency matrix of $\Gamma$ can be written as
$$
A_{\Gamma}= \begin{pmatrix}
A & B \\
-B^{\mathsf{T}} & C
\end{pmatrix}.
$$
Observe that $B$ is precisely the adjacency matrix of the Cayley color graph $\mathrm{Cay}(G, f)$ determined by the indicating function
$$
f(g) = \begin{cases}
1, & \text{if } g\in S,\\[0.2cm]
0, & \text{otherwise}.
\end{cases}
$$

Recall that for every character $\chi_k$ ($k \in G$) of the abelian group $G$, the character table $P$ diagonalizes any $G$-circulant matrix.
Thus, by the spectral decomposition of Cayley graphs, the matrices $A, B$ and $C$ are simultaneously diagonalized by the character table $P$ of $G$:
\begin{equation}\label{rls}
\begin{matrix}
 \Lambda_{R}=\mathrm{diag}(\chi_{g}(R)-\chi_{g}(R^{-1}))_{g\in G}=P^{H}AP,\\[0.2cm]
 \Lambda_{L}=\mathrm{diag}(\chi_{g}(L)-\chi_{g}(L^{-1}))_{g\in G}=P^{H}CP,\\[0.2cm]
 \Lambda_{S}=\mathrm{diag}(\chi_{g}(S))_{g\in G}=P^{H}BP.
\end{matrix}
\end{equation}
Set
$$
D=\begin{pmatrix}
P  & 0 \\
0  & P
\end{pmatrix},
$$
which is a unitary matrix. Using \eqref{rls}, we obtain
$$
D^{H}A_{\Gamma}D=
\begin{pmatrix}
P^{H}  & 0 \\
0  & P^{H}
\end{pmatrix}
\begin{pmatrix}
A  & B \\
-B^H  & C
\end{pmatrix}
\begin{pmatrix}
P  & 0 \\
0  & P
\end{pmatrix}=
\begin{pmatrix}
 \Lambda_{R}  &  \Lambda_{S} \\
 -(\Lambda_{S})^{H}  &  \Lambda_{L}
\end{pmatrix}=
\widehat{A_{\Gamma}}.
$$
Since $D$ is unitary, $A_{\Gamma}$ and $\widehat{A}_{\Gamma}$ share the same eigenvalues. Notice that $\Lambda_{R}$, $\Lambda_{L}$ and $\Lambda_{S}$ are diagonal matrices. Hence, after a suitable permutation of the rows and columns, $\widehat{A_{\Gamma}}$ can be brought into a block-diagonal form consisting of $n$ blocks of size $2\times 2$, each of the form
\begin{equation}\label{Ai}
A_{k}=\begin{pmatrix}
  \chi_{k}(R)-\chi_{k}(R^{-1})&   \chi_{k}(S)\\[0.2cm]
  -\overline{\chi_{k}(S)}  &   \chi_{k}(L)-\chi_{k}(L^{-1})
\end{pmatrix}.
\end{equation}
Consequently, the characteristic polynomial of $A_{\Gamma}$ factors as
$$
\det(\lambda I - \widehat{A}_{\Gamma}) = \prod_{k\in G} \det(\lambda I - A_{k}).
$$

\begin{lemma}\label{Specsemi-Cayley}
Let $\Gamma = \mathrm{OSC}(G,R,L,S)$ be an oriented semi-Cayley graph over an abelian group $G$ of order $n$. Denote
$$
r_{k} =\chi_{k}(R)-\chi_{k}(R^{-1}),~~ l_k =\chi_{k}(L)-\chi_{k}(L^{-1}).
$$
\begin{itemize}
\item[\rm(a)] The eigenvalues of $A_{\Gamma}$ are
\begin{equation}\label{lam+-}
\lambda_{k}^{\pm}=\frac{1}{2}\left(r_{k}+l_k \pm \mathrm{i}\left |\sqrt{(l_k-r_{k})^2-4\left | \chi_{k}(S) \right |^2 }  \right | \right), ~~~k\in G.
\end{equation}
\item[\rm(b)] For each $k\in G$, let $v_{k}^{\pm}=(a_{k}^{\pm},b_{k}^{\pm})^{\mathsf{T}}$ be the normalized eigenvectors of $A_{k}$ corresponding to $\lambda_{k}^{\pm}$, where $\ast^{\mathsf{T}}$ denotes the transpose of $\ast$. Then the eigenprojectors of $A_{\Gamma}$ associated with $\lambda_{k}^{\pm}$ are
\begin{equation}\label{EigenSEMI}
E_{\lambda_{k}}^{\pm}=
\begin{pmatrix}
\left|a_{k}^{\pm}\right|^{2}  & a_{k}^{\pm}\overline{b_{k}^{\pm}} \\[0.2cm]
 b_{k}^{\pm}\overline{a_{k}^{\pm}} & \left|b_{k}^{\pm}\right|^{2}
\end{pmatrix}\otimes \widetilde{E}_k, ~~~k\in G,
\end{equation}
where $\widetilde{E}_{k}$ is defined in \eqref{EX}. In particular, if $\chi_{k}(S)=0$, then we use the convention $\lambda_k^+=r_k$, $\lambda_k^-=l_k$, and choose
$$
v_{k}^{+}=(1,0)^{\mathsf{T}},\qquad v_{k}^{-}=(0,1)^{\mathsf{T}}.
$$
In this case,
\begin{equation}\label{chik0}
E_{\lambda_{k}}^{+}=
\begin{pmatrix}
1  & 0 \\
0 & 0
\end{pmatrix}\otimes \widetilde{E}_k, \text{~and~~} E_{\lambda_{k}}^{-}=
\begin{pmatrix}
0  & 0 \\
0 & 1
\end{pmatrix}\otimes \widetilde{E}_k.
\end{equation}
\end{itemize}
Moreover, the transition matrix of the quantum walk on $\Gamma$ (with respect to $A_{\Gamma}$) can be written as
\begin{equation}\label{USEMI}
U(t)=\sum_{k\in G}\sum_{\pm}\exp(\lambda_{k}^{\pm}t)E_{\lambda_{k}}^{\pm}.
\end{equation}
\end{lemma}
\begin{proof}
(a) From the block decomposition of $\widehat{A}_{\Gamma}$, the spectrum of $A_{\Gamma}$, counted with multiplicities, is the union of the spectra of the $2\times2$ blocks $A_k$ ($k\in G$). Computing the characteristic polynomial of $A_k$ defined in \eqref{Ai} and solving for its roots yields exactly the expression~\eqref{lam+-}. Hence, (a) follows.

(b) Fix $k\in G$ and let $v_{k}^{\pm}=(a_k^{\pm}, b_k^{\pm})^{\mathsf{T}}$ be the normalized eigenvectors of $A_k$ for $\lambda_k^{\pm}$. Following the argument of \cite[Lemma 11]{Arez}, the corresponding eigenvectors of $A_{\Gamma}$ are given by
$$
\widetilde{v}_k^{\pm} = \frac{1}{\sqrt{n}}  (a_k^{\pm},b_k^{\pm} )^{\mathsf{T}}\otimes \left( \chi_k(g_1),\dots,\chi_k(g_n) \right)^{\mathsf{T}} .
$$
Since the vectors $\tilde{v}_k^\pm$ form an orthonormal eigenbasis of $A_{\Gamma}$, a direct computation yields the block Kronecker product formula in \eqref{EigenSEMI}.
\qed
\end{proof}

Let $v_{k}^{\pm}=(a_{k}^{\pm},b_{k}^{\pm})^{\mathsf{T}}$ be as in Lemma \ref{Specsemi-Cayley} (b). Define
\begin{equation}\label{DdefineCDE}
c_{k}^{\pm}=|a_{k}^{\pm}|^{2},\quad
d_{k}^{\pm}=|b_{k}^{\pm}|^{2},\quad
e_{k}^{\pm}=a_{k}^{\pm}\overline{b_{k}^{\pm}}.
\end{equation}

Suppose that $u=(g,r)$ and $v=(h,s)$ are two vertices in $\Gamma$, where $r,s\in \{0,1\}$. By \eqref{EigenSEMI} and \eqref{USEMI}, we have
\begin{equation}\label{UENTRY}
U(t)_{u, v}=\left\{\begin{array}{ll}
\frac{1}{n} \sum_{k \in G}\left(c_{k}^{+} \exp \left(t \lambda_{k}^{+}\right)+c_{k}^{-} \exp \left(t \lambda_{k}^{-}\right)\right) \chi_{k}\left(g h^{-1}\right), & r=s=0, \\\\
\frac{1}{n} \sum_{k \in G}\left(d_{k}^{+} \exp \left(t \lambda_{k}^{+}\right)+d_{k}^{-} \exp \left(t \lambda_{k}^{-}\right)\right) \chi_{k}\left(g h^{-1}\right), & r=s=1, \\\\
\frac{1}{n} \sum_{k \in G}\left(e_{k}^{+} \exp \left(t \lambda_{k}^{+}\right)+e_{k}^{-} \exp \left(t \lambda_{k}^{-}\right)\right) \chi_{k}\left(g h^{-1}\right), & r=0, s=1, \\\\
\frac{1}{n} \sum_{k \in G}\left(\overline{e_{k}^{+}} \exp \left(t \lambda_{k}^{+}\right)+\overline{e_{k}^{-}} \exp \left(t \lambda_{k}^{-}\right)\right) \chi_{k}\left(g h^{-1}\right), & r=1, s=0.
\end{array}\right.
\end{equation}

\begin{lemma}\label{CPMsum}
Let $c_{k}^{\pm}, d_{k}^{\pm}, e_{k}^{\pm}$ be as in \eqref{DdefineCDE}. Then
$$
c_{k}^{+}+c_{k}^{-}=d_{k}^{+}+d_{k}^{-}=1, ~e_{k}^{+}+e_{k}^{-}=0.
$$
In particular, if $\chi_{k}(S)=0$, then
$$c_{k}^+=d_{k}^-=1,~c_{k}^-=d_{k}^+=e_{k}^{+}=e_{k}^{-}=0.$$
\end{lemma}

\begin{proof}
Notice that $v_{k}^{\pm}=(a_{k}^{\pm},b_{k}^{\pm})^{\mathsf{T}}$ are the orthogonal eigenvectors of $A_{k}$ corresponding to $\lambda_{k}^{\pm}$. Then $A_{k}$ has the  spectral decomposition
\begin{equation}\label{Akspectral}
A_{k}=\sum_{\pm}\lambda_{k}^{\pm}
\begin{pmatrix}
\left|a_{k}^{\pm}\right|^{2}  & a_{k}^{\pm}\overline{b_{k}^{\pm}} \\[0.2cm]
 b_{k}^{\pm}\overline{a_{k}^{\pm}} & \left|b_{k}^{\pm}\right|^{2}
 \end{pmatrix}
=\sum_{\pm}\lambda_{k}^{\pm}
\begin{pmatrix}
c_{k}^{\pm}  & e_{k}^{\pm}  \\[0.2cm]
\overline{e_{k}^{\pm}}  & d_{k}^{\pm}
 \end{pmatrix}.
\end{equation}
Since $A_{k}$ is a normal matrix, the sum of all eigenprojectors of $A_{k}$ is equal to the identity matrix. It follows from \eqref{Akspectral} that $c_{k}^{+}+c_{k}^{-}=d_{k}^{+}+d_{k}^{-}=1,~ e_{k}^{+}+e_{k}^{-}=0$.

If $\chi_{k}(S)=0$, by Lemma \ref{Specsemi-Cayley} (b) and \eqref{DdefineCDE}, one can easily verify the result.

This completes the proof.
\qed
\end{proof}

\begin{theorem}\label{SemiA}
Let $\Gamma=\mathrm{OSC}(G,R,L,S)$ be an oriented semi-Cayley graph over an abelian group of order $n$. Then $\Gamma$ does not admit $(\alpha,\beta)$-FR with $\alpha,\beta\neq0$ from $(g,r)$ to $(h,r)$, for $r\in\{0,1\}$.
\end{theorem}

\begin{proof}
Assume that $\Gamma$ admits an $(\alpha,\beta)$-FR with $\alpha,\beta\neq0$ from $(g,r)$ to $(h,r)$. From \eqref{EX} and \eqref{EigenSEMI}, we see that $(g,r)$ and $(h,r)$ are cospectral. Applying Theorem~\ref{FR=PST}, we obtain
\begin{equation}\label{rh}
U(t)_{(g,r),(g,r)}=U(t)_{(h,r),(h,r)}=\alpha,~~~U(t)_{(h,r),(g,r)}=-U(t)_{(g,r),(h,r)}=\beta.
\end{equation}
By \eqref{UENTRY} and \eqref{rh}, for every $k\in G$, we have
\begin{equation}\label{*}
U(t)_{(gk,r),(gk,r)}=\alpha,~~~U(t)_{(hk,r),(gk,r)}=\beta.
\end{equation}
Taking $k=g^{-1}h$ in $\eqref{*}$, we have
\begin{equation}\label{hh}
U(t)_{(h,r),(h,r)}=\alpha,~~~U(t)_{(g^{-1}h^2,r),(h,r)}=\beta.
\end{equation}

Assume that $g^{-1}h^{2}\neq g$. Then $(g,r)$, $(h,r)$ and $(g^{-1}h^{2},r)$ are three distinct vertices. From \eqref{rh} and \eqref{hh}, we have
$$
\left |  U(t)_{(g,r),(h,r)} \right | ^2+\left | U(t)_{(h,r),(h,r)} \right | ^2+\left |  U(t)_{(g^{-1}h^2,r),(h,r)}\right | ^2=\alpha^2+2\beta^2>1,
$$
since $\beta\neq0$ and $\alpha^{2}+\beta^{2}=1$. This contradicts the fact that $U(t)$ is unitary. Thus we must have $g^{-1}h^{2}=g$. Substituting this into \eqref{hh} and comparing with \eqref{rh}, we obtain
$$
U(t)_{(g,r),(h,r)}=-\beta=U(t)_{(g^{-1}h^2,r),(h,r)}=\beta,
$$
which forces $\beta=0$, contradicting the assumption $\beta\neq0$.

This completes the proof.
\qed
\end{proof}

\begin{theorem}\label{FRmatrix}
Let $\Gamma=\mathrm{OSC}(G,R,L,S)$ be an oriented semi-Cayley graph over an abelian group of order $n$. If $\Gamma$ admits $(\alpha,\beta)$-FR with $\alpha,\beta\neq0$ from $(g,0)$ to $(h,1)$ at time $t$, then the transition matrix $U(t)$ can be written in the block form
\begin{equation}\label{UT0}
U(t)=\begin{pmatrix}
 \alpha I & -\beta Q^{\mathsf{T}}W \\
 \beta Q &  \alpha W
\end{pmatrix},
\end{equation}
where $Q$ is a permutation matrix, and $W$ is an $n\times n$ orthogonal matrix.
\end{theorem}
\begin{proof}
Assume that $\Gamma$ has $(\alpha,\beta)$-FR  with $\alpha,\beta\neq0$ from $(g,0)$ to $(h,1)$ at time $t$. Then
\begin{equation}\label{FRsingle}
U(t) \mathbf{e}_{(g,0)} = \alpha \mathbf{e}_{(g,0)} + \beta \mathbf{e}_{(h,\,1)}.
\end{equation}
Combining \eqref{UENTRY} with \eqref{FRsingle}, for every $z\in G$, we get
$$
U(t) \mathbf{e}_{(z,0)} = \alpha \mathbf{e}_{(z,0)} + \beta \mathbf{e}_{(g^{-1}h z,\,1)}.
$$
Arrange the rows of $U(t)$ such that the first $n$ rows correspond to vertices $\{(x,0)\mid x\in G\}$ and the last $n$ rows to $\{(x,1)\mid x\in G\}$.
From the above description we immediately obtain the two left blocks of $U(t)$:
$$
\text{upper-left block} = \alpha I_n,\qquad
\text{lower-left block} = \beta Q,
$$
where $Q$ is the $n\times n$ permutation matrix defined by $Q_{x,z}=1$ if $x=g^{-1}hz$, and 0 otherwise.

Let
\begin{equation}\label{UT1}
U(t) =
\begin{pmatrix}
\alpha I_n & A' \\
\beta Q & W'
\end{pmatrix},
\end{equation}
where $A',W'\in\mathbb{R}^{n\times n}$. Recall that $U(t)$ is real orthogonal. Then
\begin{align*}
\begin{pmatrix}
\alpha I_n & \beta Q^{\mathsf{T}} \\
(A')^{\mathsf{T}} & (W')^{\mathsf{T}}
\end{pmatrix}
\begin{pmatrix}
\alpha I_n & A' \\
\beta Q & W'
\end{pmatrix}
&=
\begin{pmatrix}
\alpha^2 I_n + \beta^2 Q^{\mathsf{T}}Q & \alpha A' + \beta Q^{\mathsf{T}}W' \\
\alpha (A')^{\mathsf{T}} + \beta (W')^{\mathsf{T}}Q & (A')^{\mathsf{T}}A' + (W')^{\mathsf{T}}W'
\end{pmatrix} = I_{2n}.
\end{align*}
Thus, we have
\begin{align}
\alpha A' + \beta Q^{\mathsf{T}} W' &= 0,\label{U01}\\
(A')^{\mathsf{T}}A' + (W')^{\mathsf{T}}W' &= I_n.\label{U11}
\end{align}
Rewrite \eqref{U01} as $A' = -\frac{\beta}{\alpha}Q^{\mathsf{T}}W'$. Substituting this into \eqref{U11} yields
$$
\frac{\beta^2}{\alpha^2} (W')^{\mathsf{T}}QQ^{\mathsf{T}}W' + (W')^{\mathsf{T}}W' = I_n .
$$
Recall that $QQ^{\mathsf{T}} = I_n$. Hence, we have $(W')^{\mathsf{T}}W' = \alpha^2 I_n$, which shows that $\frac{1}{\alpha}W'$ is an orthogonal matrix. Set $W = \frac{1}{\alpha}W'$. Then $W' = \alpha W$ and $A' = -\beta Q^{\mathsf{T}} W$. Therefore, $U(t)$ has exactly the block form given in \eqref{UT0}.
\qed
\end{proof}

Let $G$ be a finite abelian group of order $n$ with irreducible characters set $\widehat{G}$. For any function $f: G \to \mathbb{C}$, its \emph{Fourier transform} $\widehat{f}: \widehat{G} \to \mathbb{C}$ is defined by
$$
\widehat{f}(\chi) = \sum_{g \in G} f(g) \overline{\chi(g)}, \quad \chi \in \widehat{G}.
$$
The \emph{Fourier inversion} \cite{BFourier} is given by
$$
f(g) = \frac{1}{n} \sum_{\chi \in \widehat{G}} \widehat{f}(\chi) \chi(g), \quad g \in G.
$$

\begin{theorem}\label{NONCOSPECTRAL}
Let $\Gamma=\mathrm{OSC}(G,R,L,S)$ be an oriented semi-Cayley graph over an abelian group of order $n$. Let
$$
X=\{k\in G: \chi_{k}(S)=0\}.
$$
Then $\Gamma$ admits $(\alpha,\beta)$-FR with $\alpha,\beta\neq0$ from $u=(g,0)$ to $v=(h,1)$ at time $t$ if and only if all the following conditions hold:
\begin{itemize}
\item[\rm(a)] $X=\emptyset$;
 \item[\rm(b)] Set $a=g^{-1}h$. For every $k\in G$,
  \begin{align*}
    \exp\left(t\lambda_k^+\right) & = \alpha + \beta c_k^{-}\left(\overline{e_k^+}\right)^{-1} \overline{\chi_k(a)},\\[0.2cm]
    \exp\left(t\lambda_k^-\right) & = \alpha - \beta c_k^{+}\left(\overline{e_k^+}\right)^{-1} \overline{\chi_k(a)},
  \end{align*}
  where $c_k^\pm, e_k^+$ are defined in \eqref{DdefineCDE}.
\end{itemize}
\end{theorem}

\begin{proof}
We first prove the necessity. Assume that $\Gamma$ has $(\alpha,\beta )$-FR with $\alpha,\beta\neq0$ from $u=(g,0)$ to $v=(h,1)$ at time $t$. Then
\begin{equation}\label{uuv}
U(t)\mathbf{e}_u = \alpha\mathbf{e}_u + \beta\mathbf{e}_v.
\end{equation}
Applying \eqref{UENTRY} to \eqref{uuv}, we have the following identities:
\begin{itemize}
  \item For $w=(f,0)$,
$$
U(t)_{w,u}=\frac{1}{n}\sum_{k\in G}\left(c_{k}^{+}\exp(t\lambda_{k}^{+})+c_{k}^{-}\exp(t\lambda_{k}^{-})\right)\chi_{k}(fg^{-1})
=\left\{\begin{array}{cl}
          \alpha, & \text{if $w=u$,} \\[0.2cm]
          0,& \text{otherwise.}
        \end{array}\right.
$$
\item For $w=(f,1)$,
$$
U(t)_{w,u}=\frac{1}{n}\sum_{k\in G}\left(\overline{e_{k}^{+}} \exp \left(t \lambda_{k}^{+}\right)+\overline{e_{k}^{-}} \exp \left(t \lambda_{k}^{-}\right)\right)\chi_{k}(fg^{-1})
=\left\{\begin{array}{cl}
          \beta, & \text{if $w=v$,} \\[0.2cm]
          0,& \text{otherwise.}
        \end{array}\right.
$$
\end{itemize}

According to the Fourier inversion, for every $k\in G$,
\begin{equation}\label{cz}
c_{k}^{+}\exp(t\lambda_{k}^{+})+c_{k}^{-}\exp(t\lambda_{k}^{-})=\alpha,
\end{equation}
\begin{equation}\label{ez}
\overline{e_{k}^{+}} \exp \left(t \lambda_{k}^{+}\right)+\overline{e_{k}^{-}} \exp \left(t \lambda_{k}^{-}\right)=\beta\overline{\chi_k(a)}.
\end{equation}

If $X\neq\emptyset$, then $\exists~k\in X$ such that $\chi_{k}(S)=0$. Lemma~\ref{CPMsum} gives that $e_{k}^{+}=e_{k}^{-}=0$. Substituting this into \eqref{ez}, we obtain $\beta \overline{\chi_k(a)}=0$, contradicting $\beta\neq0$. Hence $X=\emptyset$ and $e_{k}^{+}\neq0$. From Lemma \ref{CPMsum}, we obtain $e_{k}^{-}=-e_{k}^{+}$. Applying this to \eqref{ez} gives that
\begin{equation}\label{lampm}
\exp(t\lambda_{k}^{+})-\exp(t\lambda_{k}^{-}) = \left(\overline{e_{k}^{+}}\right)^{-1}\beta\overline{\chi_k(a)}.
\end{equation}
Recall from Lemma~\ref{CPMsum} that $c_{k}^{+}+c_{k}^{-}=1$. By \eqref{cz} and \eqref{lampm}, we get
$$
\exp(t\lambda_{k}^{+})=\alpha +c_{k}^{-}\left(\overline{e_{k}^{+}}\right)^{-1}\beta \overline{\chi_k(a)}~~~\text{and}~~~
\exp(t\lambda_{k}^{-})=\alpha -c_{k}^{+}\left(\overline{e_{k}^{+}}\right)^{-1}\beta \overline{\chi_k(a)}.
$$

Next, we show the sufficiency. For any $w=(f,0)$, \eqref{UENTRY} gives
$$
U(t)_{w,u}
= \frac1n\sum_{k\in G}\bigl(c_{k}^{+}\exp(t\lambda_{k}^{+})+c_{k}^{-}\exp(t\lambda_{k}^{-})\bigr)\chi_{k}(fg^{-1}).
$$
Using (b) and $c_{k}^{+}+c_{k}^{-}=1$, it follows from \eqref{CHIsum} that
\begin{equation}\label{UW0}
U(t)_{w,u}= \frac{\alpha}{n}\sum_{k\in G}\chi_{k}(fg^{-1})
= \begin{cases}
\alpha, & \text{if } w=u,\\[0.2cm]
0, & \text{otherwise}.
\end{cases}
\end{equation}

Now let $w=(f,1)$. Recall from Lemma \ref{CPMsum} that $e_k^- = -e_k^+$. By \eqref{UENTRY} and (b), we have
\begin{align*}
U(t)_{w,u}
&=\frac{1}{n} \sum_{k \in G}\left(\overline{e_{k}^{+}} \exp \left(t \lambda_{k}^{+}\right)+\overline{e_{k}^{-}} \exp \left(t \lambda_{k}^{-}\right)\right) \chi_{k}\left(fg^{-1}\right)\\
&=\frac{1}{n} \sum_{k \in G} \overline{e_{k}^{+}}\left(\exp \left(t \lambda_{k}^{+}\right)-\exp \left(t \lambda_{k}^{-}\right)\right) \chi_{k}\left(fg^{-1}\right)\\
&=\frac{\beta}{n}\sum_{k\in G} \chi_{k}(fg^{-1}) \overline{\chi_k(a)}.
\end{align*}
By \eqref{OR}, we have
\begin{equation}\label{UW1}
U(t)_{w,u}=\left\{\begin{array}{cl}
          \beta, & \text{if $w=v$,} \\[0.2cm]
          0,& \text{otherwise.}
        \end{array}\right.
\end{equation}

Combining \eqref{UW0} and \eqref{UW1}, we have $U(t)\mathbf{e}_u = \alpha\mathbf{e}_u + \beta\mathbf{e}_v$. Thus, $\Gamma$ admits $(\alpha,\beta )$-FR from $u=(g,0)$ to $v=(h,1)$ at time $t$.

This completes the proof.
\qed
\end{proof}

\begin{theorem}\label{SemiABcospectral}
Let $\Gamma=\mathrm{OSC}(G,R,L,S)$ be an oriented semi-Cayley graph over an abelian group of order $n$. Let
$$
X=\{k\in G: \chi_{k}(S)=0\}.
$$
Then $(g,0)$ and $(h,1)$ are strongly cospectral in $\Gamma$ if and only if all the following conditions hold:
\begin{itemize}
\item[\rm(a)] $X=\emptyset$;
\item[\rm(b)] For any $k$, $r_k=l_k$, where $r_{k}, l_{k}$ are as in Lemma \ref{Specsemi-Cayley};
\item[\rm(c)] For any $k_1,k_2\in G$ and $\sigma,\tau\in\{+,-\}$, if $\lambda_{k_1}^{\sigma}=\lambda_{k_2}^{\tau}$, then
$$
\sigma \frac{\overline{\chi_{k_1}(S)}}{|\chi_{k_1}(S)|}\chi_{k_1}(g^{-1}h)
=\tau \frac{\overline{\chi_{k_2}(S)}}{|\chi_{k_2}(S)|}\chi_{k_2}(g^{-1}h).
$$
\end{itemize}
\end{theorem}
\begin{proof} We first prove the necessity. Let  $u=(g,0)$ and $v=(h,1)$. If $X\neq\emptyset$, then we take $k\in G$ such that $\chi_{k}(S)=0$. By Lemma \ref{Specsemi-Cayley} (b), the corresponding eigenprojectors are given by \eqref{chik0}. By \eqref{chik0}, we have
$$
E_{\lambda_k}^{+}\mathbf e_u\neq 0,\qquad
E_{\lambda_k}^{+}\mathbf e_v=0,
$$
Hence, there does not exist a phase factor $\exp{\left(\mathrm{i}\pi q_{k}^{+}(u,v)\right)}$ such that
$$
E_{\lambda_{k}}^{+}\mathbf e_{u}=\exp{\left(\mathrm{i}\pi q_{k}^{+}(u,v)\right)}E_{\lambda_{k}}^{+}\mathbf e_{v}.
$$
Therefore, $u$ and $v$ are not strongly cospectral. This proves (a).

Let $(a_{k}^{\pm},b_{k}^{\pm})^{\mathsf{T}}$ be the normalized eigenvectors of $A_{k}$ corresponding to $\lambda_{k}^{\pm}$. By \eqref{EX} and \eqref{EigenSEMI}, for any $k$, we have
\begin{equation}\label{EAB01}
E_{\lambda_{k}}^{\pm}\mathbf e_{u}=\overline{ \left(a_{k}^{\pm}/b_{k}^{\pm}\right) } \chi_{k}(g^{-1}h) E_{\lambda_{k}}^{\pm}\mathbf e_{v}.
\end{equation}
Notice that $u$ and $v$ are strongly cospectral. Denote the phase factor associated with $\lambda_{k}^{\pm}$ by $\exp{\left(\mathrm{i}\pi q_{k}^{\pm}(u,v)\right)}$. Combining \eqref{Scospectral} with \eqref{EAB01}, we get
\begin{equation}\label{abchi11}
\left|\overline{ \left(a_{k}^{\pm}/b_{k}^{\pm}\right) } \chi_{k}(g^{-1}h)\right| =\left|\exp{\left(\mathrm{i}\pi q_{k}^{\pm}(u,v)\right)}\right|=1.
\end{equation}
 By \eqref{Ai}, we get
$$
\begin{pmatrix}
  r_k&   \chi_{k}(S)\\[0.2cm]
  -\overline{\chi_{k}(S)}  &  l_k
\end{pmatrix}
\begin{pmatrix}
  a_{k}^{\pm}\\[0.2cm]
 b_{k}^{\pm}
\end{pmatrix}=
\begin{pmatrix}
 r_k a_{k}^{\pm}+\chi_{k}(S)b_{k}^{\pm}\\[0.2cm]
   -\overline{\chi_{k}(S)}a_{k}^{\pm}+l_k b_{k}^{\pm}
\end{pmatrix}=
\lambda_{k}^{\pm}
\begin{pmatrix}
  a_{k}^{\pm}\\[0.2cm]
 b_{k}^{\pm}
\end{pmatrix}.
$$
Then
\begin{equation}\label{aPM}
a_{k}^{\pm}=\frac{\chi_{k}(S)}{\lambda_{k}^{\pm}-r_{k}}b_{k}^{\pm}.
\end{equation}
Note that $\left | \chi_{k}(g^{-1}h) \right | =1$. By \eqref{lam+-}, \eqref{abchi11} and \eqref{aPM}, we have
\begin{equation}\label{l-r}
\left |l_{k}-r_k+\mathrm{i}\left |\sqrt{(l_{k}-r_k)^2-4\left | \chi_{k}(S) \right |^2 }  \right | \right |
=\left |l_{k}-r_k-\mathrm{i}\left |\sqrt{(l_{k}-r_k)^2-4\left | \chi_{k}(S) \right |^2 }  \right | \right |
=\left |2\chi_{k}(S) \right |.
\end{equation}
Note that $r_{k}$ and $l_{k}$ are the eigenvalues of oriented graph $\mathrm{Cay}(G, R)$ and $\mathrm{Cay}(G, L)$, respectively. Then $l_k - r_k$ is either zero or a nonzero purely imaginary number. Thus, $(l_k - r_k)^2\le 0$. Notice that $\left | \chi_{k}(S) \right |^2>0$ since $X=\emptyset$. Then
$$
(l_{k}-r_k)^2-4\left | \chi_{k}(S) \right |^2<0.
$$
Hence, by \eqref{l-r}, we have $l_k - r_k=0$, that is, $r_{k}=l_{k}$, yielding (b).

Note that (a) and (b) imply $\chi_{k}(S)\neq 0$, $r_k=l_k$ for any $k\in G$. It follows from \eqref{lam+-} that
$$
\lambda_{k}^{\pm}=r_{k}\pm \mathrm{i} \left | \chi_{k}(S) \right |.
$$
Substituting this into \eqref{aPM}, we get
\begin{equation}\label{bapm1}
a_{k}^{\pm}=\mp  \mathrm{i}  \frac{\chi_{k}(S)}{\left | \chi_{k}(S) \right |}b_{k}^{\pm}.
\end{equation}
Substituting \eqref{bapm1} into \eqref{EAB01}, we have
\begin{align}\label{Epmi}
E_{\lambda_{k}}^{\pm}\mathbf e_{u}
&=\pm \mathrm{i}\frac{\overline{\chi_{k}(S)}}{\left | \chi_{k}(S) \right |}\chi_{k}(g^{-1}h)E_{\lambda_{k}}^{\pm}\mathbf e_{v}.
\end{align}
Note that $u$ and $v$ are strongly cospectral. For any $k_1,k_2\in G$ and $\sigma,\tau\in\{+,-\}$, if $\lambda_{k_1}^{\sigma}=\lambda_{k_2}^{\tau}$, then by \eqref{laml=lam2} and \eqref{Epmi}, we get
$$
\sigma \frac{\overline{\chi_{k_1}(S)}}{|\chi_{k_1}(S)|}\chi_{k_1}(g^{-1}h)
=\tau \frac{\overline{\chi_{k_2}(S)}}{|\chi_{k_2}(S)|}\chi_{k_2}(g^{-1}h).
$$
Thus, (c) holds.

Next, we prove the sufficiency. From (a), we know that for any $k\in G$, $\chi_{k}(S)\neq 0$. By Lemma \ref{Specsemi-Cayley} and (b), for any $k\in G$, $\sigma\in\{+,-\}$, we get \eqref{Epmi}.
Note that for any $k\in G$,
$$
\left|\pm \mathrm{i}  \frac{\overline{ \chi_{k}(S)}}{\left | \chi_{k}(S) \right |}\chi_{k}(g^{-1}h)\right|=1.
$$
Then by \eqref{Epmi}, for each $\lambda_{k}^\sigma$, there exists a phase factor $\exp{\left(\mathrm{i}\pi q_{k}^{\sigma}(u,v)\right)}$ such that
\begin{align}\label{Exp=pmchi}
\exp{\left(\mathrm{i}\pi q_{k}^{\sigma}(u,v)\right)} =\sigma \mathrm{i}  \frac{\overline{ \chi_{k}(S)}}{\left | \chi_{k}(S) \right |}\chi_{k}(g^{-1}h).
\end{align}
Hence,
$$
E_{\lambda_k}^\sigma\mathbf e_u = \exp\bigl(\mathrm{i}\pi q_k^\sigma(u,v)\bigr)\,E_{\lambda_k}^\sigma\mathbf e_v .
$$
If $\lambda_{k_1}^\sigma = \lambda_{k_2}^\tau$, by (c) and \eqref{Exp=pmchi}, we have
$$
\exp{\left(\mathrm{i}\pi q_{k_1}^{\sigma}(u,v)\right)}=\exp{\left(\mathrm{i}\pi q_{k_2}^{\tau}(u,v)\right)}.
$$
Thus, $(g,0)$ and $(h,1)$ are strongly cospectral in $\Gamma$.

This completes the proof.
\qed
\end{proof}

\begin{theorem}\label{SemiAB}
Let $\Gamma=\mathrm{OSC}(G,R,L,S)$ be an oriented semi-Cayley graph over an abelian group of order $n$. Suppose that $u=(g,0)$ and $v=(h,1)$ are strongly cospectral in $\Gamma$. Then $\Gamma$ admits $(\alpha,\beta)$-FR with $\alpha,\beta\neq0$ from $u=(g,0)$ to $v=(h,1)$ at time $t$ if and only if all the following conditions hold:
\begin{itemize}
\item[\rm(a)] $G$ can be written as $K_1\cup K_2$, where
 \begin{align*}
  K_1&=\left\{k_1 : \overline{\chi_{k_1}(S)}\chi_{k_1}(g^{-1}h)=+\left | \chi_{k_1}(S) \right |\right\},\\
K_2&=\left\{k_2 : \overline{\chi_{k_2}(S)}\chi_{k_2}(g^{-1}h)=-\left | \chi_{k_2}(S) \right |\right\};
 \end{align*}
\item[\rm(b)] For every $k_1\in K_1$ and every $k_2\in K_2$, respectively, we have
$$
(\lambda_{k_1}^{+}-\lambda_{0}^{+})t\in 2\pi\mathrm{i}\mathbb{Z},\qquad   (\lambda_{k_2}^{-}-\lambda_{0}^{+})t\in 2\pi\mathrm{i}\mathbb{Z}.
$$
Moreover,
$$
\alpha=\mathrm{Re}\left(\exp(t\lambda_{0}^{+})\right), ~~\beta=-\mathrm{Im} \left(\exp(t\lambda_{0}^{+})\right).
$$
\end{itemize}
\end{theorem}
\begin{proof}
We first prove the necessity. Assume that $\Gamma$ has $(\alpha,\beta )$-FR with $\alpha,\beta\neq0$ from $u=(g,0)$ to $v=(h,1)$ at time $t$. Notice that $u$ and $v$ are strongly cospectral. It follows from Theorem \ref{FRNS} and \eqref{EAB01} that
\begin{equation}\label{abchi}
\overline{ \left(a_{k}^{\pm}/b_{k}^{\pm}\right) } \chi_{k}(g^{-1}h) \in \{\pm \mathrm{i}\}.
\end{equation}
Combining \eqref{bapm1} with \eqref{abchi},  we obtain
$$
\overline{\chi_{k}(S)}\chi_{k}(g^{-1}h)=\pm \left | \chi_{k}(S) \right |.
$$
Since $u$ and $v$ are strongly cospectral, Theorem \ref{SemiABcospectral} implies that $\chi_{k}(S)\neq 0$ for every $k\in G$. Thus, $G$ can be partitioned into $K_1$ and $K_2$. Hence (a) follows.

Substituting \eqref{bapm1} into \eqref{EAB01}, for any $ k_1\in K_1$, we have
\begin{align*}
E_{\lambda_{k_1}}^{\pm}\mathbf e_{u}
&=\overline{ \left(a_{k_1}^{\pm}/b_{k_1}^{\pm}\right) } \chi_{k_1}(g^{-1}h) E_{\lambda_{k_1}}^{\pm}\mathbf e_{v}\\
&=\overline{\left(\mp  \mathrm{i}  \frac{\chi_{k_1}(S)}{\left | \chi_{k_1}(S) \right |}\right)}\chi_{k_1}(g^{-1}h)E_{\lambda_{k_1}}^{\pm}\mathbf e_{v}\\
&=\pm  \mathrm{i}\frac{\overline{\chi_{k_1}(S)}}{\left | \chi_{k_1}(S) \right |}\chi_{k_1}(g^{-1}h)E_{\lambda_{k_1}}^{\pm}\mathbf e_{v}\\
&=\pm  \mathrm{i}E_{\lambda_{k_1}}^{\pm}\mathbf e_{v}.
\end{align*}
Similarly, for any $ k_2\in K_2$, we have
\begin{align*}
E_{\lambda_{k_2}}^{\pm}\mathbf e_{u} &=\mp  \mathrm{i}E_{\lambda_{k_2}}^{\pm}\mathbf e_{v}.
\end{align*}
Thus, we have
\begin{equation}\label{k1k2pm}
\begin{aligned}
\Lambda_{u,v}^+ &= \{\lambda_{k_1}^{+}\mid k_1\in K_1\}\cup \{\lambda_{k_2}^{-}\mid k_2\in K_2\},\\
\Lambda_{u,v}^- &= \{\lambda_{k_1}^{-}\mid k_1\in K_1\}\cup \{\lambda_{k_2}^{+}\mid k_2\in K_2\}.
\end{aligned}
\end{equation}
Note that
$$
\overline{\chi_0(S)}\chi_0(g^{-1}h)=|S|=\chi_{0}(S).
$$
Hence $0\in K_1$, and therefore, by \eqref{k1k2pm}, $\lambda_{0}^{+}\in \Lambda_{u,v}^+$. Denote the corresponding phase factor by $\exp{\left(\mathrm{i}\pi q_{0}^{+}(u,v)\right)}=\mathrm{i}$.
By Corollary \ref{FRNS2} and \eqref{k1k2pm}, for every $k_1\in K_1$, $k_2\in K_2$,  we have
$$
(\lambda_{k_1}^{+}-\lambda_{0}^{+})t\in 2\pi\mathrm{i}\mathbb{Z},\qquad   (\lambda_{k_2}^{-}-\lambda_{0}^{+})t\in 2\pi\mathrm{i}\mathbb{Z}.
$$
By Theorem \ref{FRNS}, we have
$$
\alpha=\mathrm{Re}\left(\exp(t\lambda_{0}^{+})\right),
$$
and
$$
\beta
= -\frac{\mathrm{Im}\left(\exp{\left(t\lambda_{0}^{+}\right) }\right)}{\exp{\left(\mathrm{i}\pi q_{0}^{+}(u,v)\right)}} \mathrm{i}
=-\frac{\mathrm{Im}\left(\exp{\left(t\lambda_{0}^{+}\right)}\right)}{\mathrm{i}} \mathrm{i}
=-\mathrm{Im}\left(\exp{\left(t\lambda_{0}^{+}\right)}\right).
$$
Thus, (b) holds.

Next, we show the sufficiency. Since $u$ and $v$ are strongly cospectral, by Theorem \ref{SemiABcospectral} we have \eqref{Epmi}. Condition (a) gives the partition $G=K_1\cup K_2$. By \eqref{Epmi}, we get \eqref{k1k2pm}. Since $\Gamma$ is an oriented graph, by Theorem \ref{Epm}, we have
\begin{equation}\label{Lam+-}
\Lambda_{u,v}^-=-\Lambda_{u,v}^+.
\end{equation}
By (b), \eqref{k1k2pm} and \eqref{Lam+-}, for any $k_1\in K_1, k_2\in K_2$, we have
\begin{equation}\label{k1k2}
\exp(t\lambda_{k_1}^{+})=\exp(t\lambda_{k_2}^{-})=\exp(-t\lambda_{k_1}^{-})=\exp(-t\lambda_{k_2}^{+}).
\end{equation}

Set $\exp(t\lambda_{k_1}^{+})=\exp(t\lambda_{k_2}^{-})=\alpha-\mathrm{i}\beta$ with $\alpha,\beta\neq0$. Then $\exp(t\lambda_{k_1}^{-})=\exp(t\lambda_{k_2}^{+})=\alpha+\mathrm{i}\beta$. Note that the sum of the eigenprojectors of $\Gamma$ equals the identity matrix. By \eqref{USEMI}, \eqref{k1k2pm} and \eqref{k1k2}, we get
\begin{equation}\label{DCC1}
\begin{aligned}
U(t)\mathbf{e}_{u}
=&\sum_{k\in G}\sum_{\pm}\exp(\lambda_{k}^{\pm}t)E_{\lambda_{k}}^{\pm}\mathbf{e}_{u}\\
=&\sum_{k_1\in K_1}\exp(\lambda_{k_1}^{+}t)E_{\lambda_{k_1}}^{+}\mathbf{e}_{u}+\sum_{k_2\in K_2}\exp(\lambda_{k_2}^{-}t)E_{\lambda_{k_2}}^{-}\mathbf{e}_{u}\\
&+\sum_{k_1\in K_1}\exp(\lambda_{k_1}^{-}t)E_{\lambda_{k_1}}^{-}\mathbf{e}_{u}+\sum_{k_2\in K_2}\exp(\lambda_{k_2}^{+}t)E_{\lambda_{k_2}}^{+}\mathbf{e}_{u}\\
=&\sum_{k_1\in K_1}(\alpha-\mathrm{i}\beta)E_{\lambda_{k_1}}^{+}\mathbf{e}_{u}+\sum_{k_2\in K_2}(\alpha-\mathrm{i}\beta)E_{\lambda_{k_2}}^{-}\mathbf{e}_{u}\\
&+\sum_{k_1\in K_1}(\alpha+\mathrm{i}\beta)E_{\lambda_{k_1}}^{-}\mathbf{e}_{u}+\sum_{k_2\in K_2}(\alpha+\mathrm{i}\beta)E_{\lambda_{k_2}}^{+}\mathbf{e}_{u}\\
=&\alpha \sum_{k_1\in K_1} E_{\lambda_{k_1}}^{+}\mathbf{e}_{u}+\beta \sum_{k_1\in K_1} E_{\lambda_{k_1}}^{+}\mathbf{e}_{v}
+\alpha \sum_{k_2\in K_2} E_{\lambda_{k_2}}^{-}\mathbf{e}_{u}+\beta \sum_{k_2\in K_2} E_{\lambda_{k_2}}^{-}\mathbf{e}_{v}\\
&+\alpha \sum_{k_1\in K_1} E_{\lambda_{k_1}}^{-}\mathbf{e}_{u}+\beta \sum_{k_1\in K_1} E_{\lambda_{k_1}}^{-}\mathbf{e}_{v}
+\alpha \sum_{k_2\in K_2} E_{\lambda_{k_2}}^{+}\mathbf{e}_{u}+\beta \sum_{k_2\in K_2} E_{\lambda_{k_2}}^{+}\mathbf{e}_{v}\\
=& \alpha\sum_{k\in G}\sum_{\pm}E_{\lambda_k}^{\pm}\mathbf e_u +\beta\sum_{k\in G}\sum_{\pm}E_{\lambda_k}^{\pm}\mathbf{e}_{v} \\
=&\alpha \mathbf{e}_{u}+\beta \mathbf{e}_{v}.
\end{aligned}
\end{equation}
Thus, $\Gamma$ admits $(\alpha,\beta)$-FR from $u$ to $v$ at time $t$.

This completes the proof.
\qed
\end{proof}

\begin{cor}\label{properFR}
Let $\mathbb Z_n$ be the cyclic group of order $n \ge5 $. Set $R=L=\emptyset$ and $S=\mathbb Z_n\setminus\{0\}$. Then $\Gamma=\mathrm{OSC}(\mathbb Z_n,\emptyset,\emptyset,S)$ admits $\left(\cos\frac{2\pi}{n},\sin\frac{2\pi}{n}\right)\text{-FR}$ from $(u,0)$ to $(u,1)$ at time $t=\frac{2\pi}{n}$.
\end{cor}
\begin{proof}
Let $\omega_{n}=\exp\left(\frac{2\pi\mathrm{i}}{n}\right)$ and write the characters of $\mathbb Z_n$ as
$$
\chi_k(x)=\omega_{n}^{kx}, ~~k,x\in \mathbb Z_n.
$$
Then we have
\begin{equation}\label{ExamproperFR}
 \chi_k(S)
=
\begin{cases}
n-1, & \text{if~}k=0,\\[0.2cm]
-1, & \text{if~} k\in \mathbb Z_n\setminus\{0\}.
\end{cases}
\end{equation}
It follows that
$$
X=\left\{k\in \mathbb Z_n:\chi_k(S)=0\right\}=\emptyset.
$$
Thus Theorem \ref{SemiABcospectral} (a) is satisfied.

Moreover, since $R=L=\emptyset$, we have $r_k=\ell_k=0$ for every $k\in \mathbb Z_n$. Then Theorem \ref{SemiABcospectral} (b) holds.

By Lemma \ref{Specsemi-Cayley}, we get \eqref{Epmi} and the eigenvalues
\begin{equation}\label{Exameigen}
\lambda_0^\pm=\pm (n-1)\mathrm{i},~~~ \lambda_k^\pm=\pm\mathrm{i}, ~~k\in \mathbb Z_n\setminus\{0\}.
\end{equation}
By \eqref{Epmi} and \eqref{ExamproperFR}, we have
\begin{equation}\label{Exampart1}
E_{\lambda_k}^\pm\mathbf e_{(u,0)} =
\begin{cases}
\pm\mathrm i\,E_{\lambda_k}^\pm\mathbf e_{(u,1)}, & k=0,\\[0.2cm]
\mp\mathrm i\,E_{\lambda_k}^\pm\mathbf e_{(u,1)}, & k\in \mathbb Z_n\setminus\{0\}.
\end{cases}
\end{equation}
Accordingly, we set
$$
K_1=\{0\} \text{~and~} K_2=\mathbb Z_n\setminus\{0\}.
$$
Then by \eqref{Exameigen} and \eqref{Exampart1}, we have
$$
\Lambda_{(u,0),(u,1)}^+ = \{\mathrm{i}(n-1),~-\mathrm{i}\},\qquad
\Lambda_{(u,0),(u,1)}^- =\{-\mathrm{i}(n-1),~\mathrm{i}\}.
$$
Note that $\Lambda_{(u,0),(u,1)}^+\cap \Lambda_{(u,0),(u,1)}^-=\emptyset$ since $n\ge5$. Thus, any two equal eigenvalues necessarily belong to the same set and thus share the same phase factor. Thus, Theorem~\ref{SemiABcospectral} (c) is satisfied. By Theorem \ref{SemiABcospectral}, $(u,0)$ and $(u,1)$ are strongly cospectral.

We now verify the conditions of Theorem \ref{SemiAB}. Since $\chi_k(u-u)=\chi_k(0)=1$ for all $k$, it follows from \eqref{ExamproperFR} that
$$
\overline{\chi_k(S)}\chi_k(u^{-1}u)
=
\begin{cases}
|\chi_k(S)|, & \text{if~} k=0,\\[0.2cm]
-|\chi_k(S)|, & \text{if~}k\in \mathbb Z_n\setminus\{0\}.
\end{cases}
$$
Therefore, $\mathbb Z_{n}=K_1\cup K_2$. This verifies Theorem \ref{SemiAB} $\rm(a)$.

Taking $t=\frac{2\pi}{n}$, for every $k_1\in K_1$ and $k_2\in K_2$, we obtain
$$
(\lambda_{k_1}^+-\lambda_0^+)t=0, \qquad(\lambda_{k_2}^{-}-\lambda_{0}^{+})t=(-\mathrm{i}-(n-1)\mathrm{i})\frac{2\pi}{n}\in 2\pi\mathrm{i}\mathbb Z.
$$
Finally,
$$
\exp(t\lambda_0^+)= \exp\left(\frac{2(n-1)\pi\mathrm{i}}{n}\right).
$$
Since $n\ge 5$, we get
$$
\mathrm{Re} \left(\exp(t\lambda_{0}^{+})\right)=\cos\frac{2\pi}{n}\neq 0, ~~\mathrm{Im} \left(\exp(t\lambda_{0}^{+})\right)=-\sin\frac{2\pi}{n}\neq 0.
$$
Thus, Theorem \ref{SemiAB} (b) holds.

Therefore, by Theorem~\ref{SemiAB}, $\Gamma=\mathrm{OSC}(\mathbb Z_n,\emptyset,\emptyset,S)$ admits $\left(\cos\frac{2\pi}{n},\sin\frac{2\pi}{n}\right)\text{-FR}$ from $(u,0)$ to $(u,1)$ at time $\frac{2\pi}{n}$.
\qed
\end{proof}

In Corollary \ref{properFR}, we show that $\Gamma=\mathrm{OSC}(\mathbb Z_n,\emptyset,\emptyset, S)$ with $S=\mathbb Z_n\setminus\{0\}$ exhibits $\bigl(\cos\frac{2\pi}{n},\sin\frac{2\pi}{n}\bigr)\textnormal{-FR}$ from $(u,0)$ to $(u,1)$ at time $\frac{2\pi}{n}$. In fact, for $R=L\neq\emptyset$, $\Gamma=\mathrm{OSC}(\mathbb Z_n,R,R,S)$ can also admit FR between $(u,0)$ and $(u,1)$ subject to appropriate conditions; see the following example.

\begin{example}\label{example-proper-FR-nonempty}
{\em Let $R=L=\{1,5,9\}$ and $S=\{4,8\}$. Then, for every $u\in\mathbb Z_{12}$, the oriented semi-Cayley graph $\Gamma=\operatorname{OSC}(\mathbb Z_{12},R,L,S)$ admits $\left(-\frac12,\frac{\sqrt3}{2}\right)\text{-FR}$ from $(u,0)$ to $(u,1)$ at time $t=\frac{2\pi}{3}$.
}
\end{example}

\begin{proof}
Let $\omega=\exp\left(\frac{2\pi\mathrm{i}}{12}\right)$. For $R=L=\{1,5,9\}$, we have
$$
\chi_k(R)=\chi_k(L)
 =\omega^k+\omega^{5k}+\omega^{9k}
 =\omega^k\left(1+\omega^{4k}+\omega^{8k}\right).
$$
Since $\omega^4$ is a primitive third root of unity,
$$
1+\omega^{4k}+\omega^{8k}
=
\begin{cases}
3, & \text{if~}3\mid k,\\[0.2cm]
0, & \text{if~}3\nmid k.
\end{cases}
$$
It follows that
$$
r_k=l_k=
\begin{cases}
6\mathrm{i}, & \text{if~}k=3,\\
-6\mathrm{i}, & \text{if~}k=9,\\
0, & \text{otherwise}.
\end{cases}
$$
Thus, Theorem \ref{SemiABcospectral} $\rm(b)$ is satisfied.

Moreover,
\begin{equation}\label{Exam:111}
 \chi_k(S)=\omega^{4k}+\omega^{8k}
=
\begin{cases}
2, & \text{if~}3\mid k,\\[0.2cm]
-1, & \text{if~}3\nmid k.
\end{cases}
\end{equation}
Thus, $\chi_k(S)\neq0$ for every \(k\in\mathbb Z_{12}\). Hence, $X=\emptyset$, and Theorem \ref{SemiABcospectral} $\rm(a)$ is satisfied.

By Lemma~\ref{Specsemi-Cayley}, we obtain \eqref{Epmi}. Note that $ \chi_k(u^{-1}u)=\chi_k(0)=1$ for every $k$. Then by \eqref{Epmi} and \eqref{Exam:111}, we get
\begin{equation}\label{Exampart2}
E_{\lambda_k}^\pm\mathbf e_{(u,0)} =
\begin{cases}
\pm\mathrm i\,E_{\lambda_k}^\pm\mathbf e_{(u,1)}, & \text{if~}3\mid k,\\[0.2cm]
\mp\mathrm i\,E_{\lambda_k}^\pm\mathbf e_{(u,1)}, & \text{if~}3\nmid k.
\end{cases}
\end{equation}
Accordingly, we set
$$
K_1=\{0,3,6,9\},\qquad K_2=\{1,2,4,5,7,8,10,11\}.
$$
Then by \eqref{lam+-} and \eqref{Exampart2}, we have
$$
\Lambda_{(u,0),(u,1)}^+ = \{2\mathrm i,\,8\mathrm i,\,-4\mathrm i,\,-\mathrm i\},\qquad
\Lambda_{(u,0),(u,1)}^- = \{-2\mathrm i,\,-8\mathrm i,\,4\mathrm i,\,\mathrm i\}.
$$
Notice that $\Lambda_{(u,0),(u,1)}^+\cap\Lambda_{(u,0),(u,1)}^-=\emptyset$. Thus Theorem~\ref{SemiABcospectral} (c) is satisfied. By Theorem \ref{SemiABcospectral}, $(u,0)$ and $(u,1)$ are strongly cospectral.

Recall that $ \chi_k(u^{-1}u)=\chi_k(0)=1$ for every $k$. By \eqref{Exam:111}, we obtain
$$
\overline{\chi_k(S)}\chi_k(u^{-1}u)
=
\begin{cases}
|\chi_k(S)|, & \text{if~}3\mid k,\\[0.2cm]
-|\chi_k(S)|, & \text{if~}3\nmid k.
\end{cases}
$$
Therefore, $\mathbb Z_{12}=K_1\cup K_2$. Theorem \ref{SemiAB} $\rm(a)$ holds.

By \eqref{lam+-}, we have
$$
\{\lambda_{k_1}^{+}\mid k_1\in K_1\}=\{2\mathrm{i},8\mathrm{i},-4\mathrm{i}\},~~~ \{\lambda_{k_2}^{-}\mid k_2\in K_2\}=\{-\mathrm{i}\},
$$
where $\lambda_{0}^{+}=2\mathrm{i}$. Taking $t=\frac{2\pi}{3}$, for every $k_1\in K_1$, $k_2\in K_2$, we have
$$
(\lambda_{k_1}^{+}-\lambda_{0}^{+})t\in 2\pi\mathrm{i}\mathbb{Z},\qquad   (\lambda_{k_2}^{-}-\lambda_{0}^{+})t\in 2\pi\mathrm{i}\mathbb{Z}.
$$
Choosing $k_1=0$, we have $\lambda_0^+=2\mathrm{i}$, and hence
$$
\exp(t\lambda_0^+)= \exp\left(\frac{4\pi\mathrm{i}}{3}\right)=-\frac12-\frac{\sqrt3}{2}\mathrm{i}.
$$
Thus, Theorem \ref{SemiAB} $\rm(b)$ is satisfied.

It follows from Theorem~\ref{SemiAB} that
$$
\alpha=-\frac12, ~~\beta=\frac{\sqrt3}{2}.
$$
Therefore, $\Gamma$ admits $\left(-\frac{1}{2},\frac{\sqrt3}{2}\right)\text{-FR}$ from $(u,0)$ to $(u,1)$ at time $\frac{2\pi}{3}$. \qed
\end{proof}

\section{Conclusion}
In this paper, we investigate the existence of proper FR on oriented graphs. We first establish necessary and sufficient conditions for oriented graphs to admit proper FR between strongly cospectral vertices. We then prove that oriented Cayley graphs over abelian groups do not admit proper FR. We further study the existence of proper FR in semi-Cayley graphs over abelian groups; such Cayley graphs over groups containing an index-two subgroup form a natural subclass of semi-Cayley graphs over that subgroup. We conclude by proposing three problems for future study:

\begin{itemize}
\item[\rm(a)] Study the existence of FR on oriented Cayley graphs over nonabelian groups.

\item[\rm(b)] Give a characterization of oriented graphs that admit FR between non-cospectral vertices.

\item[\rm(c)] Can we construct oriented graphs that admit FR between any two vertices?
\end{itemize}

\section*{Statements and Declarations}

\noindent \textbf{Competing interests}~~No potential competing interest was reported by the authors.

\medskip

\noindent \textbf{Data availability statements}~~Data sharing not applicable to this article as no datasets were generated or analysed during the current study.

\end{document}